\documentclass[11pt,reqno]{amsart}
\usepackage{amssymb}
\usepackage{amsmath}
\usepackage{geometry}                
\usepackage{graphicx}
\usepackage{epstopdf}
\usepackage{mathrsfs}
\usepackage{color}
\newcommand{\Addresses}{{
  \bigskip
  \footnotesize

 \textsc{Department of Mathematics, KTH Royal Institute of Technology,
  Stockholm, Sweden}\par\nopagebreak
  \textit{E-mail addresses:} \quad \texttt{persj@kth.se, jostromb@kth.se}
}}

\title{Convergence of sequences of Schr\"odinger means}
\author{Per Sj\"olin and Jan-Olov Str\"omberg}

\newcommand{\lowlim}{\mathop{\underline{\mathrm{lim}}}}

\renewcommand{\S}[0]{{\bf S}}

\newcommand{\R}[0]{{\mathbb R}}

\newcommand{\supp}[0]{\mbox{supp }}
\newcommand{\e}[0]{\mbox{e}}

\newtheorem{theorem}{Theorem}
\newtheorem{lemma}{Lemma}
\newtheorem{corollary}{Corollary}
\renewcommand{\smallskip}{}

\newcommand{\aname}{\em}
\newcommand{\jname}{}

\begin{document}
\begin{abstract}
We study convergence almost everywhere of sequences of Schr\"odinger means. We also replace sequences by uncountable sets.
\end{abstract}
\let\thefootnote\relax 
\maketitle
\section{Introduction}
For  $f\in L^2(\R^n), n\ge1$ and $a>0$ we set \[
\hat f(\xi)=\int_{\R^n} \e^{-i\xi\cdot x} f(x)\,dx, \xi \in \R^n,
\]
and \[ 
S_tf(x)=(2\pi)^{-n} \int_{\R^n} \e^{i\xi\cdot x} \e^{it|\xi|^a}\hat f(\xi)\, d\xi,\quad x\in\R^n, t\ge0.
\]
 For $a=2$ and $f$ belonging to the Schwartz class
$\mathscr{S}(\R^n)$ we set $u(x,t)=S_tf(x)$. It then follows that  $u(x,0)=f(x)$ and $u$ satisfies the Schr\"odinger equation 
$i\partial u/ \partial t=\Delta u$.
\par
We introduce Sobolev spaces $H_s=H_s(\R^n)$ by setting 
\[
H_s=\{f\in \mathscr{S}^\prime; \|f\|_{H_s} <\infty \}, s\in\R,
\]
where
\[
 \|f\|_{H_s}=\left(\int_{\R^n} (1+[\xi[^2)^s |\hat f(\xi)|^2\,d\xi \right)^{1/2}.
\]
In the case $a=2$ and $n=1$ it is well-known (se Carleson \cite{Car} and Dahlberg and Kenig \cite{Dah-Ken}) that 
\begin{align}
\label{eq:limit}
\lim_{t\to0}S_t f(x)=f(x)
\end{align}
almost everywhere if $f\in H_{1/4}$. Also it is known that $H_{1/4}$ cannot be replaced by $H_s$ if
$s<1/4$.
\par
In the case $a=2$ and $n>2$ Sj\"olin \cite{Sjo87} and Vega \cite{Veg88a} proved independently that (\ref{eq:limit}) holds almost everywhere if
$f \in H_s(\R^n), s>1/2$ . This result was improved by Bourgain \cite{Bou13} who proved that $f\in H_s(\R^n), s>1/2-1/4n$, is sufficient for 
convergence almost everywhere. On the other hand Bourgain \cite{Bou16} has proved that $s\ge n/2(n+1)$ is necessary for convergence for $a=2$ and
 $n\ge2$.
 \par
In the case $n=2$ and $a=2$, Du, Guth and Li  \cite{Du-Gu-Li}  proved that the condition $s>1/3$  is sufficient. Recently Du and Zhang 
\cite{Du-Zha} proved that the condition $s>n/2(n+1)$ is sufficient for $a=2$ and $n\ge3$.
\par
In the case $a>1,n=1$, (\ref{eq:limit}) holds almost everywhere if $f\in H_{1/4}$ and $H_{1/4}$  cannot be replaced by $H_s$ if $s<1/4$. In the case
$a>1, n=2$, it is known that (\ref{eq:limit}) holds almost everywhere if $f\in H_{1/2}$ and in the case $a>1,n\ge3$ convergence has been proved for $f\in H_s$ with
$s>1/2$. For the results in the case $a>1$ see Sj\"olin \cite{Sjo87,Sjo13} and Vega \cite{Veg88a,Veg88b}.
\par
If $f\in L^2(\R^n)$ then $S_tf\to f$ in $L^2$ as $t\to0$. It follows that there exists a sequence $(t_k)_1^\infty$ satisfying 
\begin{align}\
\label{eq:tlimit}
1>t_1>t_2>t_3>\dots>0 \mbox{ and }\lim_{k\to\infty}t_k=0.
\end{align}   

such that
\begin{align}
\label{eq:seqlimit}
\lim_{k\to\infty} S_{t_k} f(x)=f(x)
\end{align}
almost everywhere.
\par
In Sj\"olin \cite{Sjo19} we studied the problem of deciding for which sequences $(t_k)_1^\infty$ one has (\ref{eq:seqlimit}) almost everywhere if $f\in H_s
$. The following result was obtained in \cite{Sjo19}.\\[.2cm]
\vspace{.5cm}
{\bf Theorem A}  {\em Assume that $n\ge1$ and $a>1$ and $s>0$. We assume that  (\ref{eq:tlimit}) holds and that $\sum_{k=1}^\infty t_k^{2s/a}<\infty$ and 
$f\in H_s(\R^n)$. Then 
\[
\lim_{k\to\infty} S_{t_k} f(x)=f(x)
\]
for almost every $x\in\R^n$.
}
\par
\vspace{.2cm}
We shall here continue the study of conditions on  sequences $(t_k)_1^\infty$ which imply that (\ref{eq:seqlimit})  holds almost everywhere. We shall also replace the set
$\{t_k; k=1,2,3,\dots\}$ with sets $E$ which are not countable, for instance the Cantor set .  Our first theorem is an extension of Therorem A in which we replace the spaces $H_s$ with Bessel  potential spaces $L_s^p$. We need some more notations.\\[.2cm]
Let $1<p\le2$  and $s>0$. Set $k_s(\xi)=(1+|\xi|)^{-s/2}$ for $\xi \in \R^n$.\\
Let the operator $\mathscr{J}_s$ be defined by
\[
\mathscr{J}_s f= \mathscr{F}^{-1} (k_s\hat f), f\in L^2\cap L^p,
\]
where $\mathscr{F}$ denotes the Fourier transformation, i.e. $\mathscr{F}f=\hat f$. Then $\mathscr{J}_s$ can be extended to a bounded operator on
$L^p$, that is $k_s\in M_p$, where $M_p$ denotes the space of Fourier multipliers on $L_p$ (see Stein \cite{Ste70}, p.132).
\par
We introduce the Bessel potential space $L_s^p$ by setting $L_s^p=\{\mathscr{J}_s g; g\in L^p \},\, s>0$.\\
We let $I$ denote an interval defined in the following way. In the case $n=1, s<a/2$, and in the case $n\ge2$, we have $I=[p_0,2]$, where
$p_0=2/(1+2s/na)$. In the remaining case $n=1, s\ge  a/2$, we have $I=(1,2]$.\label{def:interval}\\
For $f\in L^p_s, p\in I$, and $a>1$, and $0<s<a$, we shall define $S_t f$ so that \[
(S_t f)\hat{}(\xi)=\e^{it|\xi |^{a} } \hat f(\xi)
\]
and then have the following theorem.
\begin{theorem} 
{\color{black}Assume $a>1,\, 0<s<a$,} and $f\in L_s^p$, where $p\in I$. Let the sequence $(t_k)_1^\infty$ satisfy (\ref{eq:tlimit}),
  and assume also that 
$\sum_{t=1}^\infty  t_k^{ps/a}<\infty$. Then
\[
\lim_{k\to\infty} S_{t_k} f(x)=f(x)
\]
almost everywhere.
\end{theorem}
In the proof of Theorem 1 we shall use the following theorem on Fourier multipliers.
\begin{theorem}
Let  $a>1, 0<s<a$, and assume also that $0<\delta<1$. Set \[
m({\xi})= \frac{\e^{ i \delta|\xi|^a}-1} {(1+|\xi|^2)^{s/2} }, \quad \xi\in\R^n. 
\]
Then $m\in M_p$ and \[
\|m\|_{M_p}\le C_p \delta^{s/a} \mbox{ for }p\in I,
\]
where $C_p$ does not depend on $\delta$.
\end{theorem}
We remark that in Sj\"olin \cite{Sjo19} we used Theorem 2 in the special case $p=2$.\\
 Now let the sequence  $(t_k)_1^\infty$ satisfy (\ref{eq:tlimit}) and set \[
 A_j=\{ t_k; 2^{-j-1}<t_k\le 2^{-j} \}\mbox{ for }j=1,2,3,\dots .
 \]
Let $\#A$ denote the number of elements in a set $A$. We have the following theorem.
\begin{theorem}
Assume that $n\ge1, a>1$, and $0<s\le1/2$ and $b\le2s/(a-s)$. Assume also that 
\begin{align}
\label{eq:nrA_est}
\#A_j\le C2^{bj}\mbox{ for }j=1,2,3,\dots
\end{align}
and that $f\in H_s$. Then
\[
\lim_{k\to\infty} S_{t_k} f(x)=f(x)
\]
almost everywhere. 
\end{theorem}
Theorem 3 has the following two corollaries.
\begin{corollary}
Assume that  $(t_k)_1^\infty$  satisfies (\ref{eq:tlimit})
and that $n\ge1, a>1, 0<s\le1/2$, and that 
$\sum_{t=1}^\infty  t_k^\gamma<\infty$, where $\gamma=2s/(a-s)$. If also $f\in H_s$ then (\ref{eq:seqlimit}) holds almost everywhere.
\end{corollary}
We remark that Corollary 1 gives an improvement of Theorem A.
\begin{corollary}
Assume that  $(t_k)_1^\infty$ statisfies (\ref{eq:tlimit}),
and that $n\ge1, \, a>1,\,1<p<2,\,r>0,\,$ and \[
s=\frac n2+r-\frac np.
\]
If $f\in L^p_r$ and $s>1/2$ then (\ref{eq:seqlimit}),
 holds almost  everywhere.\\
If $0<s\le1/2$ set $\gamma=2s/(a-s)$. If also $\sum_{t=1}^\infty  t_k^\gamma<\infty$, and $f\in L^p_r$ then (\ref{eq:seqlimit})
holds almost everywhere.
 \end{corollary}
 \vspace{.3cm}
 Now let $E$ denote a bounded set in $\R$ . For $r>0$ we let $N_E(r)$ denote the minimal number $N$ of intervals $I_l, l=1,2,\dots,N$, of length
 $r$, such that $E\subset \bigcup_1^N I_l$.\\
 For $f\in\mathscr{S} $ we introduce the maximal function \[
S^*f(x)=\sup_{t\in E}\,|S_t f(x)|\,,\quad x\in\R^n.
  \]
  We shall prove the following estimate.
  \begin{theorem} 
  Assume $n\ge1, a>0$,  and $s>0$. If $f\in\mathscr{S}$ then one has \[
  \int |S^*f(x)|^2\,dx\le C\left( \sum_{m=0}^\infty N_E( 2^{-m}) \,2^{-2ms/a} \right) \|f\|_{H_s}^2\,.
  \]
  \end{theorem}
The following corollary  follows directly
\begin{corollary}
Assume that $n\ge1, a>0, s>0, \,f\in \mathscr{S}$, and 
\begin{align}
\label{eq:densitycondition}
\sum_{m=0}^\infty N_E( 2^{-m}) \,2^{-2ms/a}<\infty.
\end{align}
Then one has
 \[
  \left(\int |S^*f(x)|^2\,dx\right)^{1/2}\le C \|f\|_{H_s}.
  \]
\end{corollary}
\vspace{.2cm}
Now let $E=\{t_k, k=1.2,3,\dots\}$ where the sequence  $(t_k)_1^\infty$ satisfies (\ref{eq:tlimit}).
We define $S^*f$ as above so that
\[
S^*f(x)=\sup_k \,|S_{t_k} f(x)|\,,\quad f\in\mathscr{S}.
 \]
 We then have the following corollary.
 \begin{corollary}
 We let $n\ge1,\,a>0, \,s>0,$ and assume that \[
 \sum_{m=0}^\infty N_E( 2^{-m}) \,2^{-2ms/a}<\infty,
 \]
 and $f\in H_s$. Then (\ref{eq:seqlimit}) holds almost everywhere.
 \end{corollary}
 \vspace{.2cm}
 Now assume $0<\kappa<1$ and that let $m_\kappa$ denote $\kappa$-dimensional Hausdorff measure on $\R$ (see
 Mattila \cite{Mat}, p.55). Let $E\subset \R$ be a Borel set with Hausdorff dimension $\kappa$ and $0<m_\kappa(E)<\infty$.
 Assume also that $0\in E$.\\
 We shall use a precise definition of $S_tf(x)$ for $f\in L^2(\R^n)$ and $(x,t)\in \R^n\times E$. Let $Q$ denote the unit cube
 $[-\frac12,\frac12 ]^n$ in $\R^n$.  Set \[
 f_N(x,t)=(2\pi)^{-n}\int_{NQ} \e^{i\xi\cdot x} \e^{it|\xi|^a}\hat f(\xi)\,d\xi,\mbox{ for } (x,t)\in\R^n\times E
 \]
 and $N=1,2,3,\dots$. It follows from well-known estimates (See Sj\"olin \cite{Sjo71} ) that there exists a set $F\subset \R^n\times E$
 with $m\times m_\kappa((\R^n\times E)\setminus F)=0$ such that
 \[ \lim_{N\rightarrow\infty} f_N(x,t)\]
 exists for every $(x,t)\in F$. Here $m$ denotes Lebesque measure. We set $S_t f(x)$ equal to this limit for
 $(x,t)\in F$  and $S_t f(x)$ will then be a measurable function on $\R^n\times E$ with respect to the measure 
 $m\times m_\kappa$\\
 Then one has the following convergence result
 \begin{theorem}
Let $n\ge1, a>0$, and assume that $s>0$ and 
\begin{align}
\label{eq:Holdersum}
\sum_{m=0}^\infty N_E( 2^{-m}) \,2^{-2ms/a}<\infty
\end{align}
and $f\in H_s$. Then for almost every $x\in\R^n$ we can modify $S_t f(x)$ on a $m_\kappa$ - nullset so that\[
\lim_{\overset{t\rightarrow0}{ t\in E}} S_t f(x) =f(x).
\]
\end{theorem}
Note that if $ 0<a<2s$ then (\ref{eq:Holdersum}) holds when $E$ is the interval $[0,1]$.
Thus one of the consequences of the above results is the following well-known fact (see Cowling \cite{Co}).
\begin{corollary}
If $0<a<2s$ and $f\in H_s$ then (\ref{eq:limit}) holds. 
\end{corollary}
We also have
\begin{corollary}
In Theorem 3 the conditions $a>1$ and $b\le2s/(a-s)$ can be replaced by the conditions
$a\ge 2s$  and   $1/b>(a-2s) / 2s $.
\end{corollary}
and
\begin{corollary}
Assume that  $(t_k)_1^\infty$  satisfies (\ref{eq:tlimit}),
  and that $n\ge1, a\ge 2s , 0<s\le1/2$, and that 
$\sum_{t=1}^\infty  t_k^\gamma<\infty$, where $1/\gamma>(a-2s)/ 2s$. If also $f\in H_s$ then (\ref{eq:seqlimit}) holds almost everywhere.
\end{corollary}

We remark that Corollary 7 gives an improvement of Theorem A and Corollary 1.\\
\par
We shall now study the case where $E$ is a Cantor set. Assume $0<\lambda<1/2$. We set $I_{0,1}=[0,1]$,
$I_{1,1}=[0,\lambda]$ and $I_{1,2}=[1-\lambda,1]$. Having defined $I_{k-1,1},\dots,I_{k-1,2^{k-1}}$, we define
$I_{k,1},\dots,I_{k,2^{k}}$ by taking away from the middle of each interval $I_{k-1,j}$ an interval of length 
$(1-2\lambda)l(I_{k-1,j})=(1-2\lambda)\lambda^{k-1}$, where $l(I)$ denotes the length of an interval $I$. We then define
Cantor sets by setting \[
C(\lambda)=\bigcap_{k=0}^\infty\bigcup_{j=1}^{2^k} I_{k,j}.
\]
It can be proved that $C(\lambda)$ has Hausdorff dimension \[\kappa=\log2/\log(1/\lambda)\] 
and that $m_\kappa(C(\lambda))=1$  (See \cite{Mat}, p. 60-62). We have the following result, where $S_t f(x)$ is defined
as in Theorem 5 with $E=C(\lambda)$.
\begin{theorem}
Assume $n\ge1, a>0,$ and $0<\lambda<1/2$. Also assume $s>a\kappa/2$ and $f\in H_s$. Then we can for almost every $x$
modify $S_t f(x)$ on  $m_\kappa$-nullset so that  \[
\lim_{\overset{t\rightarrow 0}{t\in C(\lambda)}}S_tf(x)=f(x)
.\]
\end{theorem}
\vspace{.4cm}
{\bf Remark.} In the proofs of Corollary 4 and Theorem 5 we first in the main part of the proof obtain a maximal estimate for smooth functions
and then prove a convergence result for functions 
in $H_s$. In the passage from the maximal estimate for smooth functions to the convergence result we use an approach which was 
mentioned to one of the authors by P. Sj\"ogren in a conversation, 2009.\\
\par
In Secton 2 we shall prove Theorems 1 and 2, and  Section 3 contains the proof of Theorem 3. In section 4 we prove Theorem 4, and in
Section 5 the proofs of Theorems 5 and 6 are given. \\
We shall finally construct a counter-example which gives the following theorem.
\begin{theorem}
Assume $t_k=1/(\log k)$ for $k=2,3,4,\dots$, and set \[
S^*f(x)=\sup_k|S_{t_k}f(x)|, x\in\R^n,
\]
for $f\in L^2(\R^n)$. Then $S^*$ is not a bounded operator on $L^2(R^n)$ in the case $n=1, a>1,$ and also in the case $n\ge2, a=2$.

\end{theorem}
\section{Proofs of Theorems 1 and 2}
For $m\in L^\infty(\R^n)$ and $1<p<\infty$ we set\[
T_m f=\mathscr{F}^{-1}(m\hat f), \quad f\in L^p\cup L^2.
\]
We say that $m$ is a Fourier multiplier for $L^p$ if $T_m$  can be extended to a bounded operator on $L^p$, and  we let $M_p$ denote the class of multipliers on $L^p$. We set $\| m\|_{M_p}$ equal to the norm of $T_m$ as an operator on $L^p$.
\par
Now let $1<p\le2$ and $0<s<a$. For $f\in\mathscr{S}$ and with $\hat f(\xi)=(1+|\xi|^2)^{-s/2}\hat g(\xi)$ one obtains \[
\S_t f(x)=\left(\mathscr{F}^{-1}\left(\mu(\xi)\hat g(\xi)\right)\right)(x)=T_\mu g(x),
\]
where \[
\mu(\xi)=\frac{\e^{it|\xi|^a}}{(1+|\xi|^2)^{s/2} }\,.
\]
We shall prove that $\mu\in M_p$ for $p\in I$, where $I$ is an interval defined in the introduction.
 We need som well-known results.
\begin{lemma}
Assume that $m\in M_p$ for some $p$ which $1<p<\infty$. Let $b$ be a positive number and let $k(\xi)=m(b\xi)$ for $\xi\in\R^n$. Then $k\in M_p$
and $\|k\|_{M_p}=\|m\|_{M_P}$.
\end{lemma}
We shall also use the following multiplier theorem (see Stein (\cite{Ste70}, p. 96).\\[.2cm]
{\bf Theorem B:} {\em Assume that $m$ is a bounded function on $\R^n\setminus\{0\}$ and that \[
|D^\alpha m(\xi)|\le C_\alpha |\xi|^{-|\alpha|}
\]  

for $\xi\ne0$ and $|\alpha|\le k$, where $k$ is an integer and $k>n/2$. Then $m\in M_p$ for $1<p<\infty$.
}\\
\par
We shall also need the following result (see Miyachi \cite{Miy}, p 283)\\[.2cm]
{\bf Theorem C:} {\em Assume $\psi\in C^\infty(\R^n)$ and that $\psi$ vanishes in a neighbourhood of the origin and is equal to $1$
outside a compact set. Set \[
m_{a,s}(\xi)=\psi(\xi)|\xi|^{-s} \e^{i|\xi|^a} \,,\quad\xi \in\R^n, 
\] 
where $a>1$ and $0<s<a$. Then $m_{a,s}\in M_p$ if $1<p<\infty$ and $|1/p\,-\,1/2|\le s / na$. 
}
\par
{\bf Remark. }In Miyachi's formulation of this result the function $\psi$ is replaced by a function $\psi_1$ with the properties that $\psi_1\in C^\infty,
0\le\psi_1\le1, \psi_1(\xi)=0$ for $|\xi|\le1$, and $\psi_1(\xi)=1$ for $|\xi|\ge2$. However, the two formulations are equivalent since the function
$(\psi-\psi_1)|\xi|^{-s} \e^{i|\xi|^a}$ belongs to $C^\infty_0$.\\[.2cm]
It follows from Theorem C that $m_{a,s}\in M_p$ if $p\in I$.\\
We shall then give the proof of the above statement about the function $\mu$.
\begin{lemma}
Assume $a>1$ and $0<s<a$ and also $t>0$. Set\[
\mu(\xi)=\e^{it|\xi|^a}(1+|\xi|^2)^{-s/2}\,,\quad\xi\in \R^n.
\]
Then $\mu\in M_p$ for $p\in I$.
\end{lemma}
\begin{proof}[\em Proof of Lemma 2]
 We first take $\psi$ as in Theorem C and also set $\varphi=1-\psi$. One then has
\[
\mu(\xi)=\varphi(\xi)\e^{it|\xi|^a}(1+|\xi|^2)^{-s/2} + \psi(\xi)\e^{it|\xi|^a}(1+|\xi|^2)^{-s/2} =\mu_1(\xi) + \mu_2(\xi).
 \]
We write $\mu_2=\mu_3\mu_4$, where
\[
\mu_3(\xi)=\psi(\xi)\frac{\e^{it|\xi|^a}}{|\xi|^s }
\]
and
\[
\mu_4(\xi)=\frac{|\xi|^s }{(1+|\xi|^2)^{s/2}}.
\]
We have \[
\mu_3( t^{-1/a}\eta)\,=\,\psi(t^{-1/a}\eta)\,\frac{\e^{i|\eta|^a}}{|t^{-1/a}\eta|^s}=\psi(t^{-1/a}\eta)\,t^{s/a}\,\frac{\e^{i|\eta|^a}}{|\eta|^s}\,.
\]
We let $p\in I$ and it then follows from the Remark after Theorem C that $\mu_3\in M_p$. Also $\mu_4\in M_p$ since $I\subset(1,\infty)$ 
(see Stein \cite{Ste70},  p. 133).\\
Finally\[
\mu_1(\xi)=\varphi(\xi)\,\frac{\e^{it|\xi|^a}}{(1+|\xi|^2)^{s/2}}
\]
and it is easy to see that $\mu_1$ satisfies the conditions in Theorem B. We conclude that $\mu_1\in M_p$ and thus also $\mu\in M_p$. 
\par
\end{proof}
For $f\in L^p_s, \, p\in I$, and $a>1$, and $0<s<a$, we define $S_t f$ by setting $S_t f=T_\mu g$. It is then easy to see that
\[
\left(S_t f\right)\hat{}\,(\xi)\,=\,\e^{it|\xi|^a}\hat f(\xi).
\]
Observe that according to the Hausdorff-Young theorem $\hat f\in L^q$ where $1/p\,+\,1/q\,=\,1$ .  
\par
We shall then give the proof of Theorem 2. We shall write $A\lesssim B$ if there is a constant $K$ such that $A\le KB$.\\
\begin{proof}[Proof of Theorem 2] 
We set $C=\delta^{-1/a}$  and then have $C^{-s}=\delta^{s/a}$. It follows 
that \[
m(C\xi)=\frac{\e^{i|\xi|^a} -1}{(1+C^2|\xi|^2)^{s/2}}=m_1(\xi)+m_2(\xi) - m_3(\xi),
\]
where
\[
m_1(\xi)=\varphi(\xi)\,\frac{\e^{i|\xi|^a}-1}{(1+C^2|\xi|^2)^{s/2}},
\]\[
m_2(\xi)=\psi(\xi)\,\frac{\e^{i|\xi|^a}}{(1+C^2|\xi|^2)^{s/2}}
\]
and 
\[
m_3(\xi)=\psi(\xi)\,\frac{1}{(1+C^2|\xi|^2)^{s/2}}.
\]
Here  $\varphi$  and $\psi$ are defined as in the proof of Lemma 2, and we may assume that $\varphi$  and $\psi$ are radial functions.\\
We have \[
m_2(\xi)=m_4(\xi)\,m_5(\xi),
\]
where \[
m_4(\xi)=\psi(\xi)\,\frac{\e^{i|\xi|^a}}{(C^2|\xi|^2)^{s/2}}=\delta^{s/a}\,\psi(\xi)\,\frac{\e^{i|\xi|^a}}{|\xi|^s}
\]
and \[
m_5(\xi)=\frac{(C^2|\xi|^2)^{s/2}}{(1+C^2|\xi|^2)^{s/2}}.
\]
It follows from Theorem C that $m_4\in M_p$ and $\|m\|_{M_p}\lesssim\delta^{s/a}$
for $p\in I$. Also $m_5$ has the same multiplier norm as the function $|\xi|^s(1+|\xi|^2)^{-s/2}$. We conclude that $\|m_2\|_{M_p}
\lesssim\delta^{s/a}$ for $p\in I$.\\

We want to show that 
\[
|D^\alpha m_1(\xi)|\lesssim C^{-s} |\xi|^{-|\alpha|} \mbox{ for } \xi\in\R^n\setminus\{0\}
 \]  
 for all multi-index $\alpha=(\alpha_1,\dots,\alpha_n)$, where $\alpha_i$ are non-negative integers.
 Invoking Theorem B we conclude  that \[ \| m_1\|_{M_p}\lesssim C^{-s}=\delta^{s/a}\]
 for $1<p<\infty$.\
 \par
 First we set \[
 m_{10}(x)=\varphi_0(x) \frac{\e^{ix^{a/2}}-1}{(1+C^2x)^{s/2}},
 \] 
 where we define $\varphi_0$ by taking $\varphi_0(x)=\varphi(\xi)$ if $x=|\xi|^2$ and we then have
 $m_1(\xi)=m_{10}(|\xi|^2)$.
 
 We get for  $x>0$  \[
 D^j\frac1{(1+C^2x))^{s/2}}=\frac{C_j C^{2j}}{(1+C^2x))^{s/2+j}}.
\]
 Hence we have
\begin{align}
\label{eq:diffest1}
|D^j\frac1{(1+C^2x)^{s/2}}| \lesssim  x^{-j} \,C^{-s}x^{-s/2}.
\end{align}
on support of $\varphi_0$.
One also has $|\e^{ix^{a/2}}-1|\le x^{a/2}$ and $D^j(\e^{ix^{a/2}}-1)$ are linear combinations of functions
$\e^{ix^{a/2}} x^{ka/2-j}$ for $j\ge1$, where $k$ is an integer $1\le k\le j$. Hence 
\begin{align}
\label{eq:diffest2}
|D^j(\e^{ix^{a/2}}-1)| \lesssim x^{a/2\,-\,j},\quad j=0,1,2,\dots,
\end{align}
for $x\in \mbox{supp}\,\varphi$.\\
A combination of (\ref{eq:diffest1}) and (\ref{eq:diffest2}) then gives \[
|D_jm_{10}(x)|\lesssim   x^{-j} C^{-s} x^{a/2 \,-\,s/2} 
\]
 Let $\alpha$  and $\beta$ denote $n$-dimenisonal muti-index. By induction over $j=0,1,2,\dots,$ and $|\alpha|=j$
 we can write $D^\alpha m_1(\xi)$ as a finite linear combination of functions of the form \[
 D^k m_{10}(|\xi|^2 ) \xi^\beta
\]
with   $ j/2 \le k \le j$ and $|\beta| =2k- j$. We conclude that \[
|D^\alpha m_1(\xi)| \lesssim \max_{|\alpha|/2\le k \le |\alpha|} |\xi|^{-2k} C^{-s} |\xi|^{a-s}|\xi|^{2k-j}=C^{-s}|\xi|^{-|\alpha|}
\lesssim \delta^{s/a}|\xi|^{-|\alpha|}.
\] 
\par
It remains to study $m_3$.  Define $m_{30}(x)$ analogously to the definition of $m_{10}(x)$ on $\mbox{supp}\,\varphi_0$, such that  $m_{30}(x)=m_3(\xi)$
when $x=|\xi|^2$, we have \[
m_{30}(x)=\psi_0(x)\,\frac{1}{(1+C^2x)^{s/2}}
\]
and invoking (\ref{eq:diffest1}) 
 \[
|D^j(1+C^2x)^{-s/2}|\lesssim C^{-s}\,x^{-j}
\]
on $\mbox{supp}\,\psi_0$. Also $|D^j\psi_0(x)|\lesssim x^{-j}$ on $\mbox{supp}\,\psi_0$.\\
We conclude that\[
|D^j(m_{30}(x)|\lesssim C^{-s}\,x^{-j}
\]
and arguing as above we obtain \[
|D^\alpha m_3(\xi)| \lesssim \max_{|\alpha||/2\le k \le |\alpha|} |\xi|^{-2k} C^{-s} |\xi|^{2k-j}=C^{-s} |\xi|^{-|\alpha|}
 \lesssim \delta^{s/a} |\xi|^{-|\alpha|}
  \]
 for $\xi\in\mbox{ supp } m_3$  and $j=0,1,2,\dots.$ Invoking Theorem B we conclude that $\|m_3\|_{M_p}\lesssim\delta^{s/a}$ 
for $1<p<\infty$. This completes the proof of Theorem 2
\end{proof}
We shall finally give the proof of Theorem 1.
\begin{proof}[Proof of Theorem 1]
We set \[
\mu_0(\xi)=\frac{\e^{it_k|\xi|^a}}{(1+|\xi|^2)^{}s/2}
\]
\[
m(\xi)=\frac{\e^{it_k|\xi|^a}-1}{(1+|\xi|^2)^{s/2}}
\]
and also have \[
k_s(\xi)=(1+|\xi|^2)^{-s/2}.
\]
It follows that \[
T_{\mu_0}g - \mathscr{J}_sg=T_mg
\]
for $g\in\mathscr{S}$.\\
We have $f\in L^p_s$ where $p\in I$ and it follows that $f=\mathscr{J}_sg$ for some $g\in L^p$. We choose a sequence $(g_j)_1^\infty$
such that $g_j\in\mathscr{S}$ and $g_j\rightarrow g$ in $L^p$ as $j\rightarrow\infty$.\\
One then has\[
T_{\mu_0}g_j - \mathscr{J}_sg_j=T_mg_j
\]
for every $j$. Letting $j$ tend to $\infty$ we obtain
\[
T_{\mu_0}g - \mathscr{J}_sg=T_mg
\]
since the three operators $T_{\mu_0}$, $\mathscr{J}_s$ and $T_m$ are all bounded on $L^p$. It follows that\[
S_{t_k}f-f=T_mg.
\]
Here we have used Lemma 2 and Theorem 2.\\
We now set $h_k=S_{t_k}f-f$ and hence $h_k=T_mg$. It follows from Theorem 2 that \[
\|h_k\|_p\lesssim t_k^{s/a} \|g\|_p
\] 
and we conclude that\[
\sum_{k=1}^\infty\int |h_k|^p\,dx\le\left( \sum_{k=1}^\infty t_k^{ps/a}  \right)\int|g|^p\,dx<\infty.
\]
Applying the theorem on monotone convergence on then obtain\[
\int \left( \sum_1^\infty|h_k|^p\right)\,dx<\infty
\]
and hence $\sum_1^\infty|h_k|^p$ is convergent almost everywhere. It follows that $\lim_{k\rightarrow\infty}|h_k|=0$ alomst everywhere and we conclude that\[
\lim_{k\rightarrow\infty}
S_{t_k}f(x)=f(x)\]
almost everywhere. This completes the proof of Theorem 1.
\end{proof}
\section{Proof of Theorem 3 and its corollaries}
We first give the proof of Theorem 3.
\begin{proof}[Proof of Theorem 3]
We may assume $b=2s/(a-s)$. Fix $j$. By adding points to $A_j$ we can get an increasing sequence $(v_k)_{k=0}^N$ and $\tilde A_j=
\{v_k;k=,0,\dots,N\}$ such that $v_0=0, v_N=2^{-j} ,  \#\tilde A_j\le C2^{bj}$, and \(
v_k-v_{k-1}\le C2^{-j}2^{-bj}.\)\\
We split the operator $S_{v_k}$ into a low frequency part and a high frequency part\[
S_{v_k}f(x)=S_{v_k,\mbox{\tiny low}_j}f(x)+S_{v_k,\mbox{\tiny high}_j}f(x)
\]
where\[
S_{k,\mbox{\tiny low}_j}f(x)0=(2\pi)^{-n} \int_{\R^n} \e^{i\xi\cdot x} \e^{iv_k|\xi|^a}\chi_{E_j}\hat f(\xi)\, d\xi,
\]
and\[
S_{k,\mbox{\tiny high}_j}f(x)=(2\pi)^{-n} \int_{\R^n} \e^{i\xi\cdot x} \e^{iv_k|\xi|^a}\chi_{E_j^c}\hat f(\xi)\, d\xi,
\]
with $E_j=\{ \xi\in\R^n; |\xi|\le2^{b_1j} \}$ and $b_1=b/\, 2s$.\\
We shall prove that
\begin{align}
\label{eq:lowsum}
\sum_j2^{bj}\sum_{\overset{v_k\in\tilde A_j}{k>0}}\| S_{k,\mbox{\tiny low}_j}f -S_{k-1,\mbox{\tiny low}_j}f  \|_2^2\le C\|f\|_{H_s}2
\end{align}
and
\begin{align}
\label{eq:highsum}
\sum_j \sum_{v_k\in\tilde A_j}\| S_{k,\mbox{\tiny high}_j}f \|_2^2\le C\|f\|_{H_s}^2.
\end{align}
We first assume that (\ref{eq:lowsum}) and (\ref{eq:highsum}) hold. Using the Schwarz inequality we then have
\[
\sup_{v_k\in\tilde A_j} |S_{k,\mbox{\tiny low}_j} f(x)-f(x)|^2\le  \left( |S_{0,\mbox{\tiny high}_j}f(x)|+ \sum_{\overset{v_k\in\tilde A_j}{k>0}}
|S_{k,\tiny\mbox{\tiny low}_j}f(x)-S_{k-1,\mbox{\tiny low}_j}f(x)|\right)^2 
\]
\[
\le  2|S_{0,\mbox{\tiny high}_j}f(x)|^2+ C2^{bj}\sum_{\overset{v_k\in\tilde A_j}{k>0}}
|S_{k,\mbox{\tiny low}_j}f(x)-S_{k-1,\mbox{\tiny low}_j}f(x)|^2
\]
and invoking (\ref{eq:lowsum})   and (\ref{eq:highsum})
\[
\sum_j \sup_{v_k\in\tilde A_j} |S_{k,\mbox{\tiny low}_j} f(x)-f(x)|^2\le 2\sum_j |S_{0,\mbox{\tiny high}_j}f(x)|^2+ C \sum_j2^{bj}\sum_{\overset{v_k\in\tilde A_j}{k>0}} |S_{k,\mbox{\tiny low}_j}f(x)-S_{k-1,\mbox{\tiny low}_j}f(x)|^2
\]
and 
\begin{align}
\label{eq:integralsuplow}
\int\sum_j \sup_{v_k\in A_j} |S_{k,\mbox{\tiny low}_j} f(x)-f(x)|^2\,dx\le C \|f\|_{H_s}^2.
\end{align}
Using (\ref{eq:highsum}) we also obtain
\begin{align}
\label{eq:integralsuphigh}
\begin{array}{l}
\int\sum_j \sup_{v_k\in A_j} |S_{k,\mbox{\tiny high}_j} f(x)|^2\,dx\\
\quad\le \int\sum_j \sup_{v_k\in\tilde A_j} |S_{k,\mbox{\tiny high}_j} f(x)|^2\,dx\le C\|f\|^2_{H_s}.
\end{array}
\end{align}
The theorem follows from (\ref{eq:integralsuplow}) and (\ref{eq:integralsuphigh}).\\
We shall now prove (\ref{eq:lowsum}) an first observe that\[
S_{k,\mbox{\tiny low}_j}f(x)-S_{k-1,\mbox{\tiny low}_j}f(x)=(2\pi)^{-n} \int_{\R^n} \e^{i\xi\cdot x}
\left( \e^{iv_k|\xi|^a}-\e^{iv_{k-1}|\xi|^a}\right)\chi_{E_j}\hat f(\xi)\, d\xi,
\]
Applying Plancherel's theorem we obtain \[
\begin{array}{l}
\|S_{k,\mbox{\tiny low}_j}f-S_{k-1,\mbox{\tiny low}_j}f\|_2^2=C\int_{E_j}|\e^{iv_k|\xi|^a}-\e^{iv_{k-1}|\xi|^a}|^2|\hat f(\xi)|^2\,d\xi \\
\quad\le C\int_{E_j} |v_k -v_{k-1}|^2 |\xi|^{2a}|\hat f(\xi)|^2\,d\xi\le C2^{-2j}2^{-2bj}\int_{E_j}|\xi|^{2a}|\hat f(\xi)|^2\, d\xi 
\end{array}
\]
and\[
\begin{array}{l}
\sum_{j\,} 2^{bj}\sum_{\overset{v_k\in\tilde A_j}{k>0}} \|S_{k,\mbox{\tiny low}_j}f-S_{k-1,\mbox{\tiny low}_j}f\|_2^2\\
\quad\le C \sum_j 2^{-2j}\left(2^{-bj}\sum_{\overset{v_k\in\tilde A_j}{k>0}}1\right) \int_{E_j}|\xi|^{2a}|\hat f(\xi)|^2\, d\xi \\
\le C\int\left( \sum_{2^{b_{\tiny1}j}\ge|\xi| }2^{-2j} \right) |\xi|^{2a}|\hat f(\xi)|^2\,d\xi.
\end{array}
\]
The inequality $2^{b_1j}\ge|\xi|$ implies $2^j\ge|\xi|^{1/b_1}$ and thus we get  \[\sum_{2^{b_1j}\ge|\xi|} 2^{-2j}\le C|\xi|^{-2/b_1}.\]
Hence the left hand side of (\ref{eq:lowsum}) is majorized by \[ C\int|\xi|^{2a-2/b_1} |\hat f(\xi)|^2\,d\xi.\]
We have $b=2s/(a-s)$ and $b_1=1/(a-s)$ and $2a-2/b_1=2a-2(a-s)=2s$ and the inequality (\ref{eq:lowsum}) follows.
\par
To prove  (\ref{eq:highsum}) we first observe that Plancherel's theorem implies \[
\|S_{k,\mbox{\tiny high}_j}f\|^2_2\le C\int_{|\xi|\ge 2^{b_1j}} |\hat f(\xi)|^2|,d\xi.
\]
and hence\[
\begin{array}{l}
\sum_j\sum_{v_k\in\tilde A_j }\|S_{k,\mbox{\tiny high}_j}f\|^2_2\le\sum_j2^{bj}\int_{|\xi|\ge2^{b_1j}}|\hat f(\xi)|^2\,d\xi\\
\quad=\int\left(\sum_{2^{b_1j\le|\xi|}} 2^{bj}\right) |\hat f(\xi)|^2\,d\xi\le C\int|\xi|^{b/b_1} |\hat f(\xi)|^2\,d\xi.
\end{array}.
\]
Since $b=2s /(a-s)$ and $b_1=1/(a-s)$ we obtain $b/b_1=2s$ and (\ref{eq:highsum}) follows.\\
Thus the proof of Theorem 3 is complete.
  \end{proof}
  \par
  We shall then prove the two corollaries to Theorem 3.
  \begin{proof}[Proof of Corollary 1]
  Since $\sum_1^\infty t_k^\gamma$  is convergent we obtain\[
  \left(\#\{k;t_k>2^{-j-1}\}\right)2^{(-j-1)\gamma} \le\sum_{t_k>2^{-j-1}} t^\gamma\le C
  \]
  an $\#A_j\le C2^{j\gamma}$ for $j=1,2,3,\dots $. Since $\gamma=2s/(a-s)$ the corollary follows from Theorem 3.
  \end{proof}
  \begin{proof}[Proof of Corollary 2 ]
  Assume that $f\in L_r^p$,  where $1<p<2$, amd $r>0$. Also let $s=n/2\,+\,r -n/p$. Then there exists $g\in L^p$ such that
  $f=\mathscr{J}_r(g)=\mathscr{J}_s(\mathscr{J}_{r-s}g)$ and we have\[
  \frac12=\frac1p-\frac{r-s}{n}.
  \]
  It follows from the Hardy-Littlewood-Sobolev theorem that $\mathscr{J}_{r-s}g\in L^2$ and hence $f\in H_s$ 
  (see Stein \cite{Ste70}. p. 119). The corollary then follows from Theorem 3.
  \end{proof}
  \section{Proofs of Theorem 4 and its corollaries}
  In Sections 4 and 5 we assume $n\ge1$ and $a>0$. We remark that (\ref{eq:limit}) holds almost everywhere if $f\in H_s$ and $n=1, 0<a<1$,
  and $s>a/4$ or $n\ge1, a=1$ and $s>1/2)$ (se Walther \cite{Wal94},\cite{Wal98}).\\
  Before proving Theorem 4 we need some preliminary estimates. We set $B(x_0;r)=\{x;|x-x_0|\le r \}$. Using the estimate\[
 |\e^{it|\xi|^a} -e^{iu|\xi|^a} |\le|t-u|\,|\xi|^a
  \]
  and with  $A\ge1$ and $\mbox{supp} \hat f\subset B(0;A)$  we obtain by Schwarz inequality
   \begin{align}
   \label{eq:infty_est}
  \begin{array}{l}
  \|S_tf-S_uf\|_\infty\le\int_{|\xi|\le A} |t-u|\,|\xi|^a\,|\hat f(\xi)|\,d\xi\\
  \quad\le|t-u|\left(\int_{|\xi|\le A} |\xi|^{2a}\,d\xi\right)^{1/2}\left(\int|\hat f(\xi)|^2\,d\xi \right)^{1/2}\\
  \quad\le C|t-u| \left(\int_0^A r^{2a+n-1}\,dr \right)^{1/2}\,\|f\|_2\le C|t-u|\,A^{a+n/2} \|f\|_2
  \end{array}
  \end{align}
  Now assume $T=\{t_j; j=0,1,2,\dots,N\}$ where $t_j\in\R$ and $t_{j-1}<t_j$. We shall prove that 
 that if  $\mbox{supp} \hat f\subset B(0;A)$ then
 \begin{align}
 \label{eq:interval_max}
 \int\max_{t,u\in T}|S_tf(x)-S_uf(x)|^2\,dx\le C \max_{t,u\in T} |t-u|^2A^{2a} \|f\|_2^2.
  \end{align}
 Using the Schwarz inequality we obtain \[
\begin{array}{l}
 \max_{t,u\in T}|S_tf(x)-S_uf(x)|\le\sum_1^N |S_{t_i}f(x)-S_{t_{i-1}}f(x)|\\
 \quad\le\sum_1^N |t_i-t_{i-1}|^{-1/2} |S_{t_i}f(x)-S_{t_{i-1}}f(x)|\,|t_i-t_{i-1}|^{1/2}\\
 \quad\le\left( \sum_1^N |t_i-t_{i-1}|^{-1} |S_{t_i}f(x)-S_{t_{i-1}}f(x)|^2\right)^{1/2}\left( \sum_1^N |t_i-t_{i-1}|\right)^{1/2}
 \end{array}
\] 
 where the last sum equals $ \max_{t,u\in T} |t-u|$, and the Plancherel theorem gives \[
 \begin{array}{l}
 \int\max_{t,u\in T}|S_tf(x)-S_uf(x)|^2\,dx\le \left(\max_{t,u\in T} |t-u|\right) \sum_1^N |t_i-t_{i-1}|^{-1} \int |S_{t_i}f(x)-S_{t_{i-1}}f(x)|^2\,dx\\
 \quad\le \left(\max_{t,u\in T} |t-u|\right) \sum_1^N |t_i-t_{i-1}|^{-1}\int|t_i-t_{i-1}|^2|\xi|^{2a} |\hat f(\xi)|^2\,d\xi\\
 \quad\le \left(\max_{t,u\in T} |t-u|\right)^2 \int|\xi|^{2a}|\hat f(\xi)|^2\,d\xi\le C \max_{t,u\in T} |t-u|^2A^{2a}\|f\|^2_2
\end{array}
 \]
 Hence (\ref{eq:interval_max}) is proved.
 \par
We shall then prove the following lemma
\begin{lemma}
Let $I$ denote an interval of length $r$ Then 
\begin{align}
\label{eq:interval_sup}
\int \sup_{t,u\in I}|S_tf(x)-S_uf(x)|^2\,dx\le Cr^2A^{2a}\| f\|_2^2
\end{align}
if $f\in L^2(\R^n)$ and $\mbox{supp }\hat f\subset B(0;A)$.
\end{lemma}  
\begin{proof}[Proof of Lemma 3]
Asumme $I=[b,b+r]$ and let $N$ be a positive integer. Set $t_i=b+ir/N, i=0,1,2,\dots,N$, and $T=\{ t_i;i=0,1,2,\dots,N\}$.
We have\[
S_tf(x)-S_uf(x)=S_{t_i}f(x)-S_{t_j}f(x)+S_tf(x)-S_{t_i}f(x) -(S_uf(x)-S_{t_j}f(x)),
\]
where we choose $t_i$ close to $t$ and $t_j$ close to $u$. Invoking (\ref{eq:infty_est}) we obtain \[
|S_tf(x)-S_{t_i}f(x)|\le  C|t-t_i|\,A^{a+n/2} \|f\|_2\le C\frac{r}{N}A^{a+n/2} \|f\|_2=C_f \frac{r}{N},
\]
and\[
|S_uf(x)-S_{t_j}f(x)|\le  C|u-t_j|\,A^{a+n/2} \|f\|_2\le C\frac{r}{N}A^{a+n/2} \|f\|_2=C_f\frac{r}{N}.
\]
where $C_f$ depends on $f$.
It follows that\[
|S_tf(x)-S_uf(x)|\le \max_{i,j}|S_{t_i}f(x)-S_{t_j}f(x)|+C_f\frac{r}{N}.
\]
Setting $F_N(x)=\max_{i,j}|S_{t_i}f(x)-S_{t_j}f(x)|$ we obtain\[
|S_tf(x)-S_uf(x)|\le F_N(x)+C_f\frac{r}{N}
\]
Letting $N\rightarrow\infty$ we obtain\[
|S_tf(x)-S_uf(x)|\le \lowlim_{N\to\infty} F_N(x).
\]
An application of Fatou's lemma and the inequality (\ref{eq:interval_max}) then gives\[
\begin{array}{l}
\int\sup_{t,u\in I}|S_tf(x)-S_uf(x)|^2\,dx\le\int\lowlim_{N\to\infty}F_N(x)^2\,dx \\[.2cm]
\quad\le\lowlim_{N\to\infty}\int F_N(x)^2\,dx \le Cr^2 A^{2a}\|f\|^2_2
\end{array}
\]
and the lemma follows.
\end{proof}
\par
Let $I$ and $f$ have the properties in the above lemma. Then
\begin{align}
\label{eq:supI}
\int\sup_{t\in I}|S_tf(x)-f(x)|^2\,dx\le C\left( r^2A^{2a}+1 \right)\|f\|_2^2.
\end{align}
To prove (\ref{eq:supI}) we take $u_0\in I$ and observe that\[
\sup_{t\in I}|S_tf(x)-f(x)|\le \sup_{t\in I}|S_tf(x)-S_{u_0}f(x)|+|S_{u_0}f(x)|+|f(x)|
\]
and (\ref{eq:supI}) follows from Lemma 3 and Plancherel theorem.
\par
We shall then prove the following lemma
\begin{lemma}
Let $f$ have the same properties as in Lemma 3. Assume $r>0$ and set 
$I_l=[t_l-r/2,t_l+r/2], l=1,2,\dots,N$. Assume that $E$ is a set and $E\subset\bigcup_1^N I_l$. Then
\begin{align}
\label{eq:supE}
\int\sup_{t\in E}|S_tf(x)-f(x)|^2\,dx\le CN\left( r^2A^{2a}+1 \right)\|f\|_2^2.
\end{align}
\end{lemma}
\begin{proof}[Poof of Lemma 4]
The lemma follows from the inequality (\ref{eq:supI}) and the inequality\[
\sup_{t\in E}|S_tf(x)-f(x)|^2 \le \sum_{l=1}^N \sup_{t\in I_l}|S_tf(x)-f(x)|^2
\]
\end{proof}
\par
Now assume $f\in\mathscr{S}$ and write\[
f=\sum_{k=0}^\infty f_k,
\]
where $\hat f_0$ is supported in $B(0;1)$ and $\hat f_k$ has support in $\{\xi;2^{k-1}\textcolor{black}{\le}|\xi|\le 2^k\}$
for $k=1,2,3,\dots$. We shall prove the following lemma
\begin{lemma}
Let $f\in\mathscr{S}$  and  $s>0$ and \textcolor{black}{and let $E$ be a bounded set in $\R$}. Then\[
\int\sup_{t\in E}|S_tf(x)-f(x)|^2\,dx\le C\|f\|_{H_s}^2\left(\sum_{k=0}^\infty N_E(2^{-ka})2^{-2ks}\right),
\]
where $N_E(r)$ for $r>0$ denotes the minimal number $N$ of intervals $I_l, l=1,2,\dots,N$, of length $r$
such that $E\subset\sum_1^N I_l$.
\end{lemma}
\begin{proof}[Proof of Lemma 5]
With real numbers $g_k>0, k=0,1,2,\dots,$ we have\[
\begin{array}{l}
\sup_{t\in E} |S_t f(x)-f(x)|\le\sum_{k=0}^\infty \sup_{t\in E} |S_t f_k((x)-f_k(x)|\\[.2cm]
\quad=\sum_{k=0}^\infty  g_k^{-1/2} \sup_{t\in E} |S_t f(_kx)-f_k(x)| g_k^{1/2}\\[.2cm]
\le \left(\sum_{k=0}^\infty  g_k^{-1} \sup_{t\in E} |S_t f_k(x)-f_k(x)|^2 \right)^{1/2}  \left(\sum_{k=0}^\infty  g_k\right)^{1/2}
\end{array}
\] 
and invoking Lemma 4 with $r=2^{-ka}$ and $A=2^k$ we obtain\[
\int\sup_{t\in E}|S_tf(x)-f(x)|^2\,dx\le \left(\sum_{k=0}^\infty  g_k\right)\left(\sum_{k=0}^\infty  g_k^{-1} CN_E(2^{-ka})(2^{-2ak}2^{2ak}+1)\|f_k\|_2^2
\right)
\]
Choosing $g_k=N_E(2^{-ka})2^{-2ks}$ one obtains\[
\begin{array}{ll}
\int\sup_{t\in E}|S_tf(x)-f(x)|^2\,dx&\le C\left(\sum_{k=0}^\infty  g_k\right)\left( \sum_0^\infty 2^{2ks}\|f_k\|_2^2 \right)\\[.2cm]
&\le C\left( \sum_0^\infty N_E(2^{-ka})2^{-2ks}\right)\|f\|_{H_s}^2
\end{array}
\]
and the proof of the lemma is complete.
\end{proof}
\par
We shall prove Theorem 4.
\begin{proof}[Proof of Theorem 4]
Let $m$ take the values $0,1,2,\dots$. If
\begin{align}
\label{eq:m2k}
2^{-m-1}<2^{-ka}\le2^{-m}
\end{align}
for some integer $k\ge0$ then\[
N_E(2^{-ka})\le C N_E(2^{-m})
\]
and since $a>0$ there is \textcolor{black}{for any fixed $m$} only a bounded number of values of $k$ for which (\ref{eq:m2k}) holds. It follows that\[
N_E(2^{-ka})2^{-2ks} \le C N_E(2^{-m})2^{-2ms/a}.
\]
Combining this inequality with the estimate\[
\sup_{t\in E}|S_tf(x)|\le\sup_{t\in E}|S_tf(x)-f(x)|   +|f(x)|
\]
one obtains the theorem from Lemma 5
\end{proof}
\par 
Corollary 3 follows directly from Theorem 4 and we shall then prove Corollary 4.
\begin{proof}[Proof of Corollary 4]
Set $E_0=E\cup\{0\}$ and\[
 S_0^* f(x)=\sup_{E_0}|S_tf(x)|,\, x\in\R^n.
\]
It then follows from Corollary 3 that  for $f\in\mathscr{S}$ one has\[
\|S_0^*f\|_2\le C\|f\|_{H_s}.
\]
It follows that for every cube $I$ in $\R^n$ one has\[
\int_I S^*_0f(x)\,dx\le C_I \| f\|_{H_s}, f\in\mathscr{S}.
\]
Now fix $f\in H_s$ and a cube $I$. Then there exists 
 { \color{black}
a sequence $(f_j)_1^\infty$ such that $f_j\in C_0^\infty$ and\[
\|f_j-f\|_{H_s}<2^{-j},\, j=1,2,3,\dots .
\]
One then has $\|f_j-f_{j+1}\|_{H_s}<2\cdot2^{-j}$ and \[
\int_I \sup_{t\in E_0} |S_t f_j(x)-S_t f_{j+1}(x)|\,dx \le C2^{-j}.
\]
}
Hence
\begin{align}
\label{eq:sup*finite}
\sum^\infty_1 \sup_{t\in E_0}|S_t f_j(x)-S_t f_{j+1}(x)| <\infty
\end{align}
for almost every $x\in I$.\\
Then choose $x$ so that (\ref{eq:sup*finite}) holds. It follow that $S_t f_j(x)\rightarrow u_x(t)$, as $j\to\infty$, uniformly in $t\in E_0$,
where $u_x$ is a continuous function on $E_0$.\\
It is also clear that $S_tf_j\rightarrow S_t f$ in $L^2$ as $j\to\infty$, for every $t\in E_0$. Since $E_0$ is countable we can find a subsequence
$(f_{j_l})_1^\infty$ such that for almost every $x$ $S_t f_{j_l}\rightarrow S_t f(x)$ for all $t\in E_0$.\\
It follows that for almost every $x\in I$ one has $S_t f(x)=u_x(t)$ for all $t\in E_0$. Since \[ 
\lim_{\overset{t\to0}{t\in E}} u_x(t)=u_x(0)
\]
almost everywhere one also has\[
\lim_{\overset{t\to0}{t\in E}} S_t f(x)=f(x) 
\]
for almost every $x\in I$. Since $I$ is arbitrary it follows that (\ref{eq:seqlimit}) holds almost everywhere in $\R^n$.
\end{proof}
\section{Proofs of Theorems 5 and 6  and  Corollaries 6 and 7}
We shall first give the proof of Theorem 5
\begin{proof}[Proof of Theorem 5]
It follows from Corollary 3 that\[
\|S^*f\|_2\le C\|f\|_{H_s},\quad f\in\mathscr{S},
\]
where\[
S^*f(x)=\sup_{t\in E}|S_t f(x)|, x\in\R^n, f\in\mathscr{S}.
\]
Now take $f\in H_s$.\\
Let $I$ denote a cube in $R^n$. It follows that \(\int_I S^*f(x)\,dx\le C_I \|f\|_{H_s}\) for $f\in C^\infty_0$.\\
We choose {\color{black}
a sequence $(f_j)_1^\infty$ such that $f_j\in C_0^\infty$ and\[
\|f_j-f\|_{H_s}<2^{-j},\, j=1,2,3,\dots .
\]
One then has $\|f_j-f_{j+1}\|_{H_s}<C 2^{-j}$ and \[
\int \sup_{t\in E} |S_t f_j(x)-S_t f_{j+1}(x)|\,dx \le C2^{-j}.
\]
}
It follows that \[
\sum^\infty_1 \sup_{t\in E}|S_t f_j(x)-S_t f_{j+1}(x)| <\infty
\]
for almost every $x\in I$. Now choose $x$ such that the above inequality holds. We conclude that $S_t f_j(x)\rightarrow u_x(t)$, as $j\to\infty$, uniformly in $t\in E$, where $u_x$ is a continuous function on $E$.\\
On the other hand $S_tf_j\rightarrow S_t f$ in $L^2(\R^n\times E; m\times m_\kappa)$ as $j\to\infty$. Hence there is a subsequence
$(f_{j_l})_1^\infty$ such that $S_t f_{j_l}(x)\rightarrow S_t f(x)$ almost everywhere in  $\R^n\times E$ with respect to $m\times m_\kappa$.
It follows that for almost every $x\in I$ one has $S_t f(x)=u_x(t)$ for almost  all $t\in E$ with respect to $m_\kappa$. We have \[
\lim_{\overset{t\to0}{t\in E}} u_x(t)=f(x) 
\]
for almost every $x\in I$ and it follows that for almost every $x\in I$ we can modify $S_t f(x)$ on a $m_\kappa$-nullset so that\[
\lim_{\overset{t\to0}{t\in E}} S_t f(x)=f(x). 
\]
This completes the proof of Theorem 5.
\end{proof}
For the proof of Corollary 6 we need the following lemma
\begin{lemma}
Let $A_j$ be defined as in Theorem 3  satisfying\[
 \#A_j\le C 2^{bj}\mbox{ for } j=0,1,2,\dots 
\] 
for some $b>0$. Let $E=\bigcup_1^\infty A_j$ and $N_E$ be as above  then\[ 
N_E(2^{-m})\le C 2^{bm/(b+1)}
\]
\end{lemma}
\begin{proof}[Proof of Lemma 6]
Fix a $k$. We have\[
\#(\left(\bigcup_1^k A_j\right)\le C \sum_{j=1}^k 2^{bj} \le C 2^{bk}
\]
and  $\bigcup_{j=k+1 }^\infty A_j \subset \{t; 0\le t \le 2^{-k-1} \}$,
which can be covered by $2^{m-k+1}$ intervals of length $2^{-m}$. Thus \[
N_E(2^{-m}) \le 2^{m-k+1}  + C2^{bk}  \]
Choose $k$   such that  $k \le (m+1)/(b+1)<k+1$ 
We get $2^{b+1} \cdot 2^{(b+1)k} > 2^{m+1}$  and $2^{bk} \le C 2^{mb/(b+1)}$. We conclude that\[
N_E(2^{-m}) \le C 2^{bk}\le C 2^{bm/(b+1)}.
\]
This ends the proof of the Lemma 6
 \end{proof}
 We can now prove Corollary 6 by using Lemma 6 and Corollary 4
 \begin{proof} [Proof of Corollary 6]
 With $1/b>(a-2s)/2s$ as in Corollary 6 we get \[
 b/(b+1)=\frac1{(1+ 1/b)} <1 / \left(1+ \frac{a-2s}{2s}\right) = 2s/a,
 \]
and we get\[ 
\sum_1^\infty N_E(2^{-m}) 2^{-2ms/a} \le C \sum_1^\infty 2^{bm/(b+1)} 2^{-2ms/a}  \le C \sum_1^\infty 2^{m(b/(b+1) - 2s/a}) <\infty
\]
since $b/(b+1) - 2s/a<0$.\\
By Corollary 4  the Corollary 6 will follow.
 \end{proof}
The Corollary 7 will now follow by similar arguments as in the proof  Corollary 1.
\par
Finally we shall give the proof of Theorem 6.
\begin{proof}[Proof of Theorem 6]
We shall use Theorem 5 with \[ \kappa=\log 2 /(\log 1/\lambda ).
\]
For $k=0,1,2,3,\dots,$ $C(\lambda)$ can be covered by $2^k$ intervals of length $\lambda^k$\\
Let $m$ be a positive integer. Choose $k$ such that $\lambda^{k+1}<2^{-m}\le\lambda^k$.
It follows that $N_E(2^{-m}) \le  2^{k+1}$ and that \[
\left(1/\lambda\right )^k\le2^m
\] 
and \[
k\le m \frac{\log 2}{\log(1/\lambda)} =\kappa m.
\]
Hence\[
\sum_{m=1}^\infty N_E(2^{-m}) 2^{-2sm/a} \le C \sum_{m=1}^\infty 2^{\kappa m}2^{-2sm/a} < \infty,
\]
if $\kappa -2s/a<0$, i.e. $s>a\kappa/2$.
Theorem 6 follows from an application of Theorem 5.
\end{proof}
\section{Proof of Theorem 7}
We first assume $n=1$ and $a>1$. We choose a function $\varphi\in C^\infty_0(\R)$ with the property that $\varphi(\xi)=1$ for $|\xi|=a^{-1/(a-1)}$ 
and also $\varphi\ge0$. We also assume that there exists a constant $A>1$ such that $\supp  \varphi \subset \{\xi\in\R;1/A\le|\xi|\le A\}$. We then define a function $f_\nu$ by setting $\hat f_\nu(\xi)=\varphi\left(2^{-\nu}\xi\right)$ where $\nu=1,2,3,\dots$. One then has \[
\|f_\nu\|_2=c\|\hat f_\nu\|_2=c\left(\int|\varphi(2^{-\nu}\xi)|^2\,d\xi\right)^{1/2}=c\left(\int|\varphi(\eta)|^2\,d\eta\,2^\nu\right)^{1/2}=c2^{\nu/2},
\]
where $c$ denotes positive constants. Setting $\eta=2^{-\nu}\xi$ we also obtain\[
S_tf_\nu(x)=c\int \e^{i\xi x}\e^{it|\xi|^a}\varphi(2^{-\nu}\xi)\,d\xi=c2^\nu\int\e^{i2^\nu\eta x}\e^{it2^{\nu a}|\eta|^a}\varphi(\eta)\,d\eta=c2^\nu\int\e^{iF(\xi)}\varphi(\xi)\,d\xi,
\]
where $F(\xi)=t2^{\nu a}|\xi|^a+2^\nu x\xi$.\\
We then assume $C2^{-\nu}\le x\le1$ where $C$ denotes a large positive constant. It is clear that $F=G+H$, where \[
G(\xi)=2^\nu x|\xi|^a+2^\nu x\xi\]
and\[
H(\xi)=t2^{\nu a}|\xi|^a -2^\nu x|\xi|^a. 
\]
We shall first study the integral  \[
\int \e^{iG(\xi)}\varphi(\xi)\,d\xi=\int \e^{i2^\nu xK(\xi)}\varphi(\xi)\,d\xi,
\]
where $K(\xi)=|\xi|^a+\xi$ for $\xi\in\R$.\\
For $\xi>0$ we have $K^\prime(\xi)=a\xi^{a-1}+1$ and for $\xi<0$ one has $K^\prime(\xi)=1-a|\xi|^{a-1}$. It follows that $K^\prime(\xi)=0$ for $\xi=-a^{-1/(a-1)}$. Also $K^{\prime\prime}(\xi)\ne0$
for $\xi\in\supp \varphi$. 
 We now apply the method of stationary phase (se Stein \cite{Ste93}, p. 334). One  obtains\[
 \left|\int \e^{iG}\varphi\,d\xi\right| \gtrsim (2^\nu x)^{-1/2}=2^{-\nu/2}x^{-1/2.}
  \]
 Hence
 \begin{align}
  \left|\int \e^{iF}\varphi\,d\xi\right| &=\left|\int \e^{i(G+H)}\varphi\,d\xi\right|=\left|\int \e^{iG}\varphi\,d\xi + \int (\e^{iG+H}-\e^{iG})\varphi\,d\xi\right| \nonumber \\[.2cm]
  &\gtrsim2^{-\nu/2}x^{-1/2}- O\left( \int \left|\e^{iH}-1\right|\varphi\,d\xi \right) \ge 2^{-\nu/2}x^{-1/2}- O\left( \int |H|\varphi\,d\xi\right), 
 \label{eq:GH}
 \end{align}
 and we need an estimate of $H$. One obtains \[
 |H(\xi)|=|t2^{\nu a}-2^\nu x||\xi|^a\lesssim |t2^{\nu a}-2^\nu x|
 \]
 on $\supp \varphi$. We then choose $k$ such that\[
 t_{k+1}< \frac{2^\nu x}{2^{\nu a} }\le t_k
 \]
 where  we assume that $\nu$ is large. It follows that \[
 t_k\le2\frac{2^\nu x}{2^{\nu a}}\le2\frac{2^\nu}{2^{\nu a}}=2\cdot 2^{\nu(1-a)}
 \]
 and hence \[
 \log k \ge\frac12 2^{\nu(a-1)}\ge2^{\nu\epsilon}
 \]
where $\epsilon>0$. It is then easy to see that \[
k\ge\e^{2^{\nu\epsilon}}
\]
and\[
t_k-t_{k+1}\le\frac1k\le\e^{-2^{\nu\epsilon}}
\]
which implies that\[
\left|t_k-\frac{2^\nu x}{2^{\nu a}}\right|\le t_k-t_{k+1}\le\e^{-2^{\nu\epsilon}}
\]
We conclude that \[
\left|t_k2^{\nu a}-2^\nu x\right|\le2^{\nu a}\e^{-2^{\nu\epsilon}}\e^{-100\nu} 
\]
for $\nu$ large.\\
Setting $t=t_k$, invoking the inequality  (\ref{eq:GH}), and using the fact that $x\le1$, one obtains\[
 \left|\int \e^{iF}\varphi\,d\xi\right| \gtrsim2^{-v/2}x^{-1/2}-O\left(\e^{-100\nu} \right)\gtrsim2^{-v/2}x^{-1/2}.
\]
It follows that\[
\int|S^*f(x)|^2\,dx\gtrsim\int_{C2^{-\nu}}^1 2^\nu \frac1x\,dx \gtrsim2^\nu\nu
\]
for $\nu$ large.\\
We have $\|f_\nu\|_2=c2^{\nu/2}$ and we have proved that $\|S^*f_\nu\|_2\gtrsim2^{\nu/2}\nu^{1/2}$ and it follows that
$S^*$ is not a bounded operator on $L^2(\R)$.\\
\par
We shall then study the case $n\ge2$ and $a=2$.
We let $\varphi\in C^\infty_0(\R)$ be the same function as in the case $n=1$. Also let $\psi\in C^\infty_0(\R^{n-1})$ and assume that $\|\psi\|_2>0$.\\
For $x\in \R^n$ we write $x=(x_1,x^\prime)$, where $x^\prime=(x_2,x_2,\dots,x_n)$. We define $f_\nu$ by setting $\hat f\nu(\xi)=\varphi(2^{-\nu}\xi_1)\psi(\xi^\prime)$
for $v=1,2,3,\dots$ . \\
It is then easy to see that $\|f_\nu\|_2=c2^{\nu/2}$ for some constant $c$.\\
We also have 
\begin{align}
S_tf_\nu(x)&=c \int_\R\int_{\R^{n-1}}\e^{i(\xi_1x_1+\xi^\prime\cdot x^\prime)}\e^{it(\xi_1^2+|\xi^\prime|^2)}\varphi(2^{-\nu} \xi_1)\psi(\xi^\prime)\,d\xi_1d\xi^\prime\nonumber\\
&=\int_\R \e^{i(\xi_1x_1+t\xi_1^2)}\varphi(2^{-\nu}\xi_1)\,d\xi_1  \int_{\R^{n-1}}\e^{i(\xi^\prime\cdot x^\prime+t|\xi^\prime|^2 )} \psi(\xi^\prime)\,d\xi^\prime \nonumber,        
\end{align}
where $c$ denotes a constant.
Setting $\eta_1=2^{-\nu}\xi_1$ we obtain\[
S_tf_\nu(x)=c2^\nu\left(\int_\R \e^{i(t2^{2\nu}\eta_1^2+2^\nu\eta_1x_1)}\varphi(\eta_1)\,d\eta_1\right)\left(\int_{\R^{n-1}}\e^{i(\xi^\prime\cdot x^\prime+t|\xi^\prime|^2 )} \psi(\xi^\prime)\,d\xi^\prime\right).
\]
We then choose $t_k$ as an approximation for $\frac{2^\nu x_1}{2^{\nu a }}$ as in the one-dimensional case and set $t(x_1)=t_k$. It follows that \[
S_{t(x_1)}f_\nu(x)=c2^\nu I(x_1)\,J(x_1,x^\prime)
\]
where \[
I(x_1)=\int_\R \e^{i(t(x_1)2^{2\nu}\eta_1^2+2^\nu\eta_1x_1)}\varphi(\eta_1)\,d\eta_1
\]
and\[
J(x_1,x^\prime)=\int_{\R^{n-1}}\e^{i(\xi^\prime\cdot x^\prime+t(x_1)|\xi^\prime|^2 )} \psi(\xi^\prime)\,d\xi^\prime.
\]
Above we proved that $|I(x_1)|\gtrsim 2^{-\nu/2} x_1^{-1/2}$ for $C2^{-\nu}\le x_1\le1$. We also have \[
S^*f_\nu(x)\gtrsim2^\nu |I(x_1)|\,|J(x_1,x^\prime)|.\]
It follows that \[
\int_{\R^{n-1}} \left( S^*f_\nu(x)\right)^2\,dx^\prime\gtrsim2^{2\nu}|I(x_1)|^2 \int_{\R^{n-1}} |J(x_1,x^\prime)|^2\,dx^\prime,
\]
and invoking Plancherel's theorem we obtain
\begin{align}
\int_{\R^{n-1}} \left( S^*f_\nu(x)\right)^2\,dx^\prime\gtrsim2^{2\nu}|I(x_1)|^2 \int_{\R^{n-1}}|\psi(\xi^\prime)|^2\,d\xi^\prime
\nonumber\\
=c2^{2\nu}|I(x_1)|^2\gtrsim 2^{2\nu}2^{-\nu}x_1^{-1}=2^\nu x_1^{-1}\nonumber
\end{align}
for $C2^{-\nu}\le x_1\le1$.\\
We conclude that\[
\int_\R\int_{\R^{n-1}} \left( S^*f_\nu(x)\right)^2\,dx_1\,dx^\prime\gtrsim2^\nu\int_{C2^{-\nu}}^11/x_1\, dx_1\gtrsim2^\nu\nu
\]
and\[
\| S^*f_\nu\|_2\gtrsim2^{\nu/2}\nu^{1/2}.
\]
Since $\|f_\nu\|_2=c2^{\nu/2}$ it follows that $S^*$ is not a bounded operator on $L^2(\R^n)$.

\Addresses
\end{document}